\documentclass[12pt]{amsart}
\usepackage{snowshovelingalg} 
\usepackage[norelsize]{algorithm2e}

\newcommand{\wv}{\lambda} %%masses sent to warehouses
\newcommand{\wvs}{\Lambda}

\newcommand{\ov}{\bm{1}} %%vector of ones
\newcommand{\gtp}{\nabla \ti \Phi}

\newcommand{\dv}{\chi}

\newcommand{\subprop}[1]{S^{prop}_{#1}}

\newcommand{\gammaN}{\beta}

\newcommand{\dQ}{d_{G(\psi)}}

\let\mbf\mathbf

\DeclareMathOperator*{\argmin}{arg\,min}

\DeclareMathOperator{\dist}{dist}

\renewcommand{\L}{\mathcal{L}}
\renewcommand{\H}{\mathcal{H}}

\newcommand\numberthis{\addtocounter{equation}{1}\tag{\theequation}}

\begin{document}
% --------------------------------------------------------------
%                         Start here
% --------------------------------------------------------------

\title[] {Two Stage Algorithm for Semi-Discrete Optimal Transport on Disconnected Domains}
 
\author{Mohit Bansil}
\address{Department of Mathematics, Michigan State University}
\email{bansilmo@msu.edu}

\subjclass[2010]{}

\begin{abstract}

In this paper we present a two-stage algorithm to solve the semi-discrete optimal transport problem in the case where the support of the source measure is disconnected. We establish global linear convergence and local superlinear convergence. We also find convergence of the associated Laguerre cells both in measure and in the sense of Hausdorff distance.

\end{abstract}

\maketitle

\tableofcontents

\section{Introduction}

We begin by recalling the semi-discrete optimal transport problem. Let $X\subset \R^n$, $n\geq 2$ be compact and $Y:=\{y_i\}_{i=1}^N \subset \R^n$ a fixed collection of finite points, along with a \emph{cost function} $c: X\times Y\to \R$. We also fix Borel probability measures $\mu, \nu$ with $\spt \mu\subset X$ and $\spt \nu \subset Y$. Furthermore $\mu$ is absolutely continuous with respect to Lebesgue measure and we denote its density by $\rho$. We set $\wv = (\nu(y_1), \dots, \nu(y_N))$. 

We want to find a measurable mapping $T: X\to Y$ such that $T_\#\mu(E):=\mu(T^{-1}(E))=\nu(E)$ for any measurable $E\subset Y$, and $T$ satisfies
 \begin{align}\label{eqn: monge}
\int_X c(x, T(x)) d\mu = \min_{\ti T_\#\mu=\nu} \int_X c(x, \ti T(x)) d\mu.
\end{align}

In this paper our goal is to propose and show convergence of a two stage algorithm. The first stage is a regularized gradient descent which achieves global linear convergence with respect to the regularized problem. The second stage is a damped Newton algorithm similar to that of \cite{KitagawaMerigotThibert19} which has superlinear local convergence. The main difference between our algorithm and that of \cite{KitagawaMerigotThibert19} is ours doesn't require $\spt \mu$ to be connected. Without a connected support the algorithm in \cite{KitagawaMerigotThibert19} may not converge at all. 

Furthermore, we will show that the Laguerre cells, $T^{-1}(y_i)$, for the approximate optimizers constructed by our algorithm converge both in the $L^1(\mu)$ sense and in the sense of Hausdorff distance. This result is analogous to that of \cite{BansilKitagawa19b}. While the $L^1(\mu)$ convergence follows directly from the results there, the Hausdorff convergence result in \cite{BansilKitagawa19b} heavily relies on the connectedness of $\spt \mu$.

\section{Setup}

\subsection{Notations and Conventions}

Here we collect some notations and conventions for the remainder of the paper. We fix positive integers $N, M, n$ with $N, n \geq 2$ and a collection $Y:=\{y_i\}_{i=1}^N\subset \R^n$. We will write $\onevect$ to denote the vector in $\R^N$ whose components are all $1$. For any vector $V \in \R^k$, we will write its components as superscripts so $V^i$ is the $i$-th component of $V$. We reserve the notation $\norm{V}$ for the $l^2$-Euclidean norm, i.e. $\norm{V} = \sqrt{\sum_{i=1}^k \abs{V^i}^2}$. We use $\norm{V}_1, \norm{V}_\infty$ for the $l^1$ and $l^\infty$ Euclidean norms, i.e. $\norm{V}_1 = \sum_{i=1}^k \abs{V^i}$ and $\norm{V}_\infty = \max_{i\in \{1, \dots, k\}} \abs{V^i}$. 
We use $\L$ to denote $n$-dimensional Lebesgue measure and $\H^k$ to denote $k$-dimensional Hausdorff measure. Also we use the notation
 \begin{align*}
\wvs = \{ \wv \in \R^N: \wv^i \in [0,1], \sum_{i=1}^N \wv^i = 1 \}
\end{align*}
for the set of all admissible weight vectors.

We will assume that the cost function $c$ satisfies the following standard conditions:
\begin{align}
c(\cdot, y_i)&\in C^2(X), \forall i\in \{1, \ldots, N\},\label{Reg}\tag{Reg}\\
\nabla_xc(x, y_i)&\neq \nabla_xc(x, y_k),\ \forall x\in X,\ i\neq k.\label{Twist}\tag{Twist}
\end{align}
We also assume the following condition, originally studied by Loeper in \cite{Loeper09}.
\begin{defin}
	The cost, $c$, is said to satisfy \emph{Loeper's condition} if for each $i\in \{1, \ldots, N\}$ there exists a convex set $Y_i\subset \R^n$ and a $C^2$ diffeomorphism $\cExp{i}{\cdot}: Y_i\to X$ such that 
	\begin{align*}
	\forall\ t\in\R,\ 1\leq k, i\leq N,\ \{p\in Y_i\mid -c(\cExp{i}{p}, y_k)+c(\cExp{i}{p}, y_i)\leq t\}\text{ is convex}.\label{QC}\tag{QC}
	\end{align*}
	See Remark \ref{rmk: brenier solutions} below for further discussion of these conditions.
	
	We also say that a set $X\subset \R^n$ is \emph{$c$-convex} with respect to $Y$ if $\invcExp{i}{X}$ is a convex set for every $i\in \{1, \ldots, N\}$. 
\end{defin}

\begin{defin}
	For any $\psi\in \R^N$ and $i\in \{1, \ldots, N\}$, we define the \emph{$i$th Laguerre cell} associated to $\psi$ as the set
	\begin{align*}
	\Lag_i(\psi):=\{x\in X\mid c(x, y_i)+\psi^i= \min_{i} c(x,y_i) + \psi^i \}.
	\end{align*}
	We also define the function $G: \R^n\to \weightvectset$ by
	\begin{align*}
	G(\psi):=(G^1(\psi), \ldots, G^N(\psi))=(\mu(\Lag_1(\psi)), \ldots, \mu(\Lag_N(\psi))),
	\end{align*}
	and denote for any $\epsilon\geq 0$,
	\begin{align*}
	\mathcal{K}^\epsilon:=\{\psi\in \R^N\mid G^i(\psi)> \epsilon,\ \forall i\in \{1, \ldots, N\}\}.
	\end{align*}
\end{defin}

\begin{rmk}\label{rmk: brenier solutions}
	The above conditions \eqref{Reg}, \eqref{Twist}, \eqref{QC} are the same ones assumed in \cite{KitagawaMerigotThibert19} and \cite{BansilKitagawa19b}. As is also mentioned in those papers, \eqref{Reg} and \eqref{Twist} are standard conditions in optimal transport. Furthermore, \eqref{QC} holds if $Y$ is a finite set sampled from from a continuous space, and $c$ is a $C^4$ cost function satisfying what is known as the \emph{Ma-Trudinger-Wang} condition (first introduced in a strong form in \cite{MaTrudingerWang05}, and in \cite{TrudingerWang09} in a weaker form).
	
	If $\mu$ is absolutely continuous with respect to Lebesgue measure (or even if $\mu$ doesn't give mass to small sets), then \eqref{Twist} implies that the Laguerre cells are pairwise $\mu$-almost disjoint. In this case the generalized Brenier's theorem \cite[Theorem 10.28]{Villani09}, tells us that for any vector $\psi\in \R^N$, the map $T_\psi: X\to Y$ defined by $T_\psi(x)=y_i$ whenever $x\in \Lag_i(\psi)$ is a minimizer in the optimal transport problem \eqref{eqn: monge}, where the source measure is $\mu$ and the target measure is defined by $\nu=\nu_{G(\psi)}$.
\end{rmk}

We define $C_\nabla := \sup_{x \in X, y \in Y} \norm{\nabla_x c(x, y)}$. 

Throughout the rest of the paper we will we assume that $X$ is compact and $c$-convex with respect to $Y$ and that $\rho$, the density of $\mu$, is $\alpha$-H\"older continuous.

In order to obtain our convergence results we will need $\spt \mu$ to be at least locally connected in a quantitative way. First we recall a definition. 

\begin{defin}
	A non-zero measure $\mu$ on a compact set $Z \subset \R^n$ satisfies a \emph{$(q, 1)$-Poincar\'e-Wirtinger inequality} for some $1\leq q\leq \infty$ if there exists a constant $\Cpw>0$ such that for any $f\in C^1(Z)$,
	\begin{align*}
	\norm{f-\int_Z f d\mu}_{L^q(\mu)}\leq \Cpw \norm{\nabla f}_{L^1(\mu)}.
	\end{align*}
	For brevity, we will write this as ``$\mu$ satisfies a \emph{$(q, 1)$-PW inequality} on $Z$''. If $\mu$ is defined on a larger set then $Z$ then the phrase ``$\mu$ satisfies a $(q, 1)$-PW inequality on $Z$'' is understood to mean that $\mu(Z) > 0$ and the restriction of $\mu$ to $Z$ satisfies a $(q, 1)$-PW inequality.
\end{defin}

\begin{rmk}
	
We remark that some version of all of our results hold with only $(1,1)$-PW inequalities but $q > 1$ will give better constants. 
	
\end{rmk}

We now define a local version of the PW inequality that will serve as our quantitative measure of the local connectivity of $\spt \mu$.

\begin{defin}
	
A probability measure $\mu$ on a compact set $X \subset \R^n$ satisfies a \emph{local $(q, 1)$-Poincar\'e-Wirtinger inequality} if there are compact $X_i \subset X$ which are pairwise disjoint and $\mu$-almost cover $X$, i.e. $\mu(X \setminus (X_1 \cup \dots \cup X_M)) = 0$ such that $\mu$ satisfies a $(q, 1)$-PW inequality on each $X_i$ with constant $\kappa_i$.	

We define the local PW constant $\kappa$ to be the $q$-th power mean of the $\kappa_i$, i.e. $\kappa = (\frac{1}{M} \sum_{i=1}^M \kappa_i^{q})^{1/q}$. Note that $\kappa \leq \max_i \kappa_i$.
\end{defin}

When we have $M=1$ we recover the global connectivity assumption of \cite[Definition 1.3]{KitagawaMerigotThibert19}. Throughout the paper we use the notation $\dv^i = \mu(X_i)$ and the vector $\dv = (\dv^1, \dots, \dv^M) \in \R^M$. 

In order to obtain convergence of our algorithm we require that optimal map $T$, ``splits up'' the $X_i$. In order to assure this we assume that the subsets sums of $\dv$ and $\wv$ are separated. For any vector $V \in \R^k$ we introduce the notation $S_V$ for the set of subset sums of $V$, i.e
 \begin{align*}
S_V = \left \{ \sum_{a \in A } a: A \subset \{V^1, \dots, V^k\} \right \}
\end{align*}
and the notation $\subprop{V}$ for the set of proper subset sums of $V$, i.e.
\begin{align*}
\subprop{V} = \left \{ \sum_{a \in A } a: A \subsetneq \{V^1, \dots, V^k\}, A \neq \emptyset \right \}. 
\end{align*} 
With this notation, we say that the subset sums of $\dv$ and $\wv$ are separated if  $\subprop{\wv} \cap S_\dv = \emptyset$. 

\begin{rmk}\label{rmk: small set}

Note that for any fixed $X$ with fixed decomposition $X_1, \dots, X_M$, the collection of all $\wv \in \wvs$ such that our condition is not satisfied (i.e. $\subprop{\wv} \cap S_\dv \neq \emptyset$) is a ``small set'' in the sense that it has Hausdorff dimension $n-2$ whereas $\wvs$ has Hausdorff dimension $n-1$. In particular if $\wv$ is randomly chosen from $\wvs$ (with respect to any probability measure that is absolutely continuous with respect to $\H^{n-1}$), then with probability 1, we have $\subprop{\wv} \cap S_\dv = \emptyset$. We defer a formal proof of this till Proposition \ref{prop: small set}.

\end{rmk}

For the remainder of the paper we use the notation $d_{\wv} := \dist(\subprop{\wv}, S_\dv)$ for any $\wv \in \wvs$. Note that since $N > 1$, $\subprop{\wv}$ is never empty and so $d_{\wv}$ is always defined. We defer a discussion of methods to compute and/or bound $d_\wv$ to section \ref{sec: estimation of d}. 

It is well-known that the optimal transport problem has a dual problem with strong duality (we refer the reader to \cite[Chapter 1]{Villani03} for background),
 \begin{align}\label{eqn: dual}
\min_{\ti T_\#\mu=\nu} \int_X c(x, \ti T(x)) d\mu 
= \sup_{\psi \in \R^N} \int_X (\min_j c(x,y_j) + \psi^j) \dmu - \inner{\psi}{\wv}.
\end{align}
Furthermore given any dual optimizer $\psi$ the map $T_\psi$ from Remark $\ref{rmk: brenier solutions}$ is a optimal transport map for the primal problem. Because of this we are able to optimize the finite dimensional dual problem instead of the infinite dimensional primal problem. We define Kantorovich's functional as in \cite[Theorem 1.1]{KitagawaMerigotThibert19}
 \begin{align*}
\Phi(\psi) 
:= \int_X (\min_{i} c(x,y_i) + \psi^i) \dmu - \inner{\psi}{\wv}
= \sum_{i} \int_{\Lag_i(\psi)} (c(x, y_i) + \psi^i) \dmu - \inner{\psi}{\wv}. 
\end{align*}
Our algorithm will find approximate maximizers of $\Phi$. We also recall that $\nabla \Phi(\psi) = G(\psi) - \wv$. 

\subsection{Previous Results}

There are many papers that apply a Newton-type algorithm to solve semi-discrete optimal transport problems and Monge-Amp\`ere equations. In \cite{OlikerPrussner88} the authors prove local convergence of a Newton Method for solving a semi-discrete Monge-Amp{\`e}re equation. Global convergence is proved in \cite{Mirebeau15}. 

For the semi-discrete optimal transport problem the authors of \cite{KitagawaMerigotThibert19} give a Newton method similar to our Algorithm \ref{alg: damped newton}, in the case when the source measure is given by a continuous probability density and has connected support. An analogous result for the case when $\mu$ is supported on a union of simplexes is given in \cite{Merigot17}, however this paper also requires a connectedness condition on the support of $\mu$. 

In \cite{BansilKitagawa19b}, Kitagawa and the author present a Newton method to solve a generalization of the semi-discrete optimal transport problem in which there is a storage fee. Our results do not require a connectedness condition on the support of $\mu$, however we require a strict convexity assumption on the storage fee function that doesn't hold for the classical semi-discrete optimal transport problem. Furthermore we show that the Laguerre cells associated to the approximate transport maps converge quantitatively both in the sense of $\mu$-symmetric distance and in the sense of Hausdorff distance.

\subsection{Outline of the Paper}

In section \ref{sec: spectral estimates} we obtain quantitative bounds on the strong concavity of $\Phi$. In section \ref{sec: convergence newton} we exploit these bounds to obtain local convergence of a damped Newton algorithm. In section \ref{sec: convergence gradient} we obtain global convergence of a gradient descent algorithm applied to a regularized version of $\Phi$. In section \ref{sec: two stage} we discuss convergence of our two stage algorithm. In section \ref{sec: convergence cells} we obtain quantitative convergence of the associated Laguerre cells. In section \ref{sec: estimation of d} we describe some methods to compute and/or estimate the parameter $d_\wv$ from the initial data. 

\section{Spectral Estimates}\label{sec: spectral estimates}

In this section we work toward quantitative bounds on the strong concavity of $\Phi$.

To start off we recall some notation in order to remain consistent with \cite{KitagawaMerigotThibert19}.
\begin{defin}
	We will write $\interior(X_i)$ to denote the interior of the set $X_i$. Given an absolutely continuous measure $\mu=\rho dx$, where $\rho$ is continuous, and a set $A\subset X_i$ with Lipschitz boundary, we will write
	\begin{align*}
	\abs{\partial A}_{\rho, X_i}:&=\int_{\partial A\cap \interior(X_i)}\rho d \H^{n-1},\quad
	\abs{A}_\rho:=\mu(A).
	\end{align*}
\end{defin}

The next lemma is virtually identical to \cite[Lemma 7.4]{BansilKitagawa19b}. We omit the proof. 

\begin{lem}\label{lem: PW inequality bound}
	Suppose that $\mu=\rho dx$ satisfies a $(q,1)$-PW inequality on $X_i$. Then
	\begin{align*}
	\inf_{A \subset X_i} \frac{\abs{\partial A}_{\rho, X_i}}{\min(\abs{A}_\rho, \abs{X_i \setminus A}_\rho)^{1/q}} \geq \frac{1}{2^{1/q}\kappa_i},
	\end{align*}
	where the infimum is over $A\subset \interior(X_i)$ whose boundary is Lipschitz with finite $\mathcal{H}^{n-1}$-measure, and $\min(\abs{A}_\rho, \abs{X_i \setminus A}_\rho)>0$.
\end{lem}

Recall from \cite{KitagawaMerigotThibert19} that $DG$ is negative semidefinite (as it is the Hessian of a concave function). Furthermore, recall that $\onevect$ is in the kernel of $DG$, so the largest eigenvalue of $DG$ is $0$. We now work toward the following estimate on the second largest eigenvalue.

\begin{thm}\label{thm: spec bound}

Suppose that $\mu$ satisfies a local $(q,1)$-PW inequality on $X$ with constant $\kappa$. Let $\psi \in \R^N$ be fixed. Assume that the subset sums of $G(\psi)$ and $\dv$ are separated i.e., $\subprop{{G(\psi)}} \cap S_{\dv} = \emptyset $. Then the second largest eigenvalue of $DG(\psi)$ is bounded away from zero. In particular it is bounded by $-\frac{2^{3-1/q}}{C_\nabla M^{1/q}N^4\kappa}d_{G(\psi)}^{1/q} < 0$ where $d_{G(\psi)} = \dist(\subprop{{G(\psi)}}, S_\dv) > 0$. 

\end{thm}

At this point, given some $\psi\in \R^N$ we recall the construction of $W$ from \cite[Section 5.3]{KitagawaMerigotThibert19}. $W$ is the simple weighted graph whose vertex set is the collection $Y$, and for any $y_i$ and $y_j$, $i\neq j$ there is an edge between $y_i, y_j$ of weight $w_{ij}$ given by
\begin{align*}
w_{ij}:=
D_i G^j(\psi)=D_jG^i(\psi)=\int_{\Lag_{i, j}(\psi)}\frac{\rho(x)}{\norm{\nabla_x c(x, y_i)-\nabla_xc(x, y_j)}}d\mathcal{H}^{n-1}(x)
\end{align*}
where we have used the notation 
\begin{align*}
\Lag_{i, j}(\psi):=\Lag_i(\psi)\cap \Lag_j(\psi)
\end{align*}
for $i$, $j\in \{1,\ldots, N\}$.
\begin{prop}\label{prop: connected subgraph}
Under the conditions of Theorem \ref{thm: spec bound}, $W$ is connected by edges of weight at least $\frac{2^{1-1/q}}{C_\nabla M^{1/q}N^2\kappa}d_{G(\psi)}^{1/q}$. 
\end{prop}
\begin{proof}
	
	Suppose by contradiction that the proposition is false. This implies that removing all edges with weight strictly less than $\frac{2^{1-1/q}}{C_\nabla M^{1/q}N^2\kappa}d_{G(\psi)}^{1/q}$ yields a disconnected graph. In other words, we can write $W = W_1 \cup W_2$ where $W_1$, $W_2 \neq \emptyset$ and are disjoint, such that every edge connecting a vertex in $W_1$ to a vertex in $W_2$ has weight strictly less than $\frac{2^{1-1/q}}{C_\nabla M^{1/q}N^2\kappa}d_{G(\psi)}^{1/q}$.  Letting $A = \cup_{y_i \in W_1} \Lag_i(\psi)$ we see that
	\begin{align*}
	\abs{\partial A}_{\rho, X} 
	\leq \sum_{\{(i, j) \mid y_i\in W_1,\ y_j\in W_2\}} 2C_\nabla w_{ij}
	< \frac{2^{2-1/q}}{M^{1/q}N^2\kappa}d_{G(\psi)}^{1/q} \abs{W_1}\abs{W_2} 
	\leq \frac{2^{2-1/q}}{M^{1/q}N^2\kappa}d_{G(\psi)}^{1/q} \frac{N^2}{4}
	= \frac{2^{-1/q}}{M^{1/q}\kappa}d_{G(\psi)}^{1/q}.
	\end{align*}
	Now for each $i$ we claim that $\min(\abs{A \cap X_i}_{\rho}, \abs{X_i \setminus A}_{\rho}) < \frac{\kappa_i^q}{M\kappa^q}d_{G(\psi)}$. If $\abs{A \cap X_i}_{\rho} = 0$ or $\abs{X_i \setminus A}_{\rho} = 0$ then the claim is trivial. Otherwise the claim follows from Lemma \ref{lem: PW inequality bound}, after noting that $\abs{\partial A}_{\rho, X_i} \leq \abs{\partial A}_{\rho, X}$. 
	
	Now note that by construction
 	\begin{align}
	\mu(A) = \sum_{y_i \in W_1} \mu(\Lag_i(\psi)) = \sum_{y_i \in W_1} G(\psi)^i \in S_{G(\psi)}. \label{eqn: mu of A is subset sum}
	\end{align}
	On the other hand we can partition the $X_i$ into two types. Let $I$ be the index set so that $i \in I$ if and only if $\abs{X_i \setminus A}_{\rho} < \frac{\kappa_i^q}{M\kappa^q}d_{G(\psi)}$. We see that
 	\begin{align*}
	\mu(A) 
	= \sum_{i=1}^M \mu(A \cap X_i)
	= \sum_{i \in I} \mu(A \cap X_i) + \sum_{i \not\in I} \mu(A \cap X_i)
	= \sum_{i \in I} \mu(X_i) - \mu(X_i \setminus A) + \sum_{i \not\in I} \mu(A \cap X_i).
	\end{align*}
	 Hence
 	\begin{align*}
	\abs{\mu(A) - \sum_{i \in I} \mu(X_i)}
	&\leq \sum_{i\in I } \mu(X_i \setminus A) + \sum_{i \not\in I} \mu(A \cap X_i) \\
	&< \sum_{i\in I } \frac{\kappa_i^q}{M\kappa^q}d_{G(\psi)} + \sum_{i \not\in I} \frac{\kappa_i^q}{M\kappa^q}d_{G(\psi)}
	= \frac{1}{M\kappa^q}\sum_{i=1}^M {\kappa_i^q}d_{G(\psi)}
	= \dQ.
	\end{align*}
	By definition $\sum_{i \in I} \mu(X_i) \in S_\dv$. Furthermore since both $W_1$ and $W_2$ are nonempty we have $\mu(A) \in \subprop{{G(\psi)}}$, recalling \eqref{eqn: mu of A is subset sum}. Hence we see that
 	\begin{align*}
	\dQ = \dist(\subprop{{G(\psi)}}, S_\dv  ) < \dQ
	\end{align*}
	which is a clear contradiction. 
\end{proof}

At this point Theorem \ref{thm: spec bound} follows via the exact same method of proof as \cite[Theorem 8.1]{BansilKitagawa19b}. We omit the proof. 

\section{Convergence of Newton Algorithm}\label{sec: convergence newton}

First we recall the standard damped Newton Algorithm (similar to the one proposed in \cite{KitagawaMerigotThibert19}).

\begin{algorithm}[H]
	
	\DontPrintSemicolon
	
	\KwIn{A tolerance $\zeta > 0$ and an initial $\psi_0\in
		\R^N$.}
	
	\While{$\norm{\nabla  \Phi(\psi_k)}  \geq \zeta$}
	{
		\begin{description}
			\item[Step 1] Compute $\vec{d}_k = - [D^2 \Phi(\psi_k)]^{-1} (\nabla \Phi(\psi_k))$
			\item [Step 2] Determine the minimum $\ell \in \N$ such that $\psi_{k+1, \ell} :=	\psi_k + 2^{-\ell} \vec{d}_k $ satisfies
			\begin{equation*}
			\begin{aligned}
			&\norm{\nabla \Phi(\psi_{k+1, \ell})} \leq (1-2^{-(\ell+1)}) \norm{\nabla \Phi(\psi_k)}.
			\end{aligned}
			\end{equation*}
			\item [Step 3] Set $\psi_{k+1} = \psi_k + 2^{-\ell}  \vec{d}_k$ and $k\gets k+1$.
			
		\end{description}	
	}	
	\caption{Damped Newton's algorithm}
	
	\label{alg: damped newton}
\end{algorithm}

We remark that unlike the algorithm proposed in \cite{KitagawaMerigotThibert19} we do not impose a condition on the cells not collapsing. This is because we will prove local convergence for the above algorithm and within the zone of convergence the error reduction requirement will automatically imply that the cells don't collapse. 

Our main result for this section is that Algorithm \ref{alg: damped newton} converges locally with superlinear rate. 

\begin{thm}\label{thm: convergence of algo}
	
	Suppose that $\mu$ satisfies a local $(q,1)$-PW inequality on $X$. If $\norm{\nabla{\Phi(\psi_0)}} \leq \frac{d_{\wv}}{2\sqrt{N}}$ then Algorithm \ref{alg: damped newton} converges with linear rate and locally with rate $1+\alpha$. 
	
\end{thm}

We start with a simple lemma. 

\begin{lem} \label{lem: d stab}

For $\wv_1, \wv_2 \in \wvs$, we have $\abs{d_{\wv_1} - d_{\wv_2}} \leq \norm{\wv_1 - \wv_2}_1$.

\end{lem}

\begin{proof}

Let $x \in \subprop{\wv_1}$. We claim that $\dist(x, \subprop{\wv_2}) \leq \norm{\wv_1 - \wv_2}_1$. Indeed if $x = \sum_{i \in I} \lambda_1^i$ then 
 \begin{align*}
\abs{x - \sum_{i \in I} \lambda_2^i} 
= \abs {\sum_{i \in I} (\lambda_1^i - \lambda_2^i)}
\leq \sum_{i \in I} \abs{\lambda_1^i - \lambda_2^i}
\leq \norm{\wv_1 - \wv_2}_1.
\end{align*}
If we apply this to $x = \argmin_{\ti x \in \subprop{\wv_1}} \dist(\ti x, S_\dv)$ then we get that there is $y \in \subprop{\wv_2}$ so that $\abs{x-y} \leq \norm{\wv_1 - \wv_2}_1$. Hence
 \begin{align*}
\dist(\subprop{\wv_2}, S_{\dv}) 
&\leq \dist(y, S_{\dv}) \\
&\leq \abs{x-y} + \dist(x, S_{\dv}) \\
&= \dist(\subprop{\wv_1}, S_{\dv}) + \abs{x-y} \\
&\leq \dist(\subprop{\wv_1}, S_{\dv}) + \norm{\wv_1 - \wv_2}_1
\end{align*}
and so we get $d_{\wv_2} - d_{\wv_1} = \dist(\subprop{\wv_2}, S_{\dv}) - \dist(\subprop{\wv_1}, S_{\dv}) \leq \norm{\wv_1 - \wv_2}_1$. A symmetric argument proves the opposite inequality. 

\end{proof}

With the above lemma, Theorem \ref{thm: spec bound} gives us that $\Phi$ is locally well-conditioned near its maximum. 

\begin{prop}\label{prop: well conditioned}
	
Suppose that $\mu$ satisfies a local $(q,1)$-PW inequality on $X$ and $\eta < d_\wv$. On the set $\mathcal{J^\eta} = \{ \psi \in \R^N: \norm{\nabla{\Phi(\psi)}}_1 < \eta \}$ we have that $\Phi$ is uniformly $C^{2, \alpha}$ and strongly concave, except in the direction $\ov$. In particular on $\mathcal{J^\eta}$
 \begin{align*}
\norm{\Phi}_{C^{2, \alpha}} < C (d_{\wv} - \eta)^{-2} 
\end{align*}
and
 \begin{align*}
D^2 \Phi \leq -\frac{2^{3-1/q}}{C_\nabla M^{1/q}N^4\kappa}(d_{\wv} - \eta)^{1/q} P
\end{align*}
where $P$ is the orthogonal projection onto the hyperplane perpendicular to $\ov$ and $C$ is a constant depending only on $\norm{\rho}_\infty, \diam(X)$, and $c$ (the cost). 

\end{prop}

\begin{proof}

Fix $\psi \in \mathcal{J^\eta}$. Since $0 \in S_{\dv}$ and each $\wv^i \in \subprop{\wv}$, we have that each $\wv^i \geq d_{\wv}$. Furthermore since 
 \begin{align*}
\abs{G^i(\psi) - \wv^i} \leq \norm{\nabla{\Phi(\psi)}}_\infty \leq \norm{\nabla{\Phi(\psi)}}_1 < \eta
\end{align*}
we obtain that $G^i(\psi) - \wv^i > -\eta$ and so $G^i(\psi) > d_{\wv} - \eta$. In other words 
\begin{align}
\mathcal{J}^\eta \subset \mathcal{K}^{d_{\wv} - \eta}. \label{eqn: J subset of K}
\end{align}
Hence the first result follows by careful tracing through the proof of \cite[Theorem 4.1]{KitagawaMerigotThibert19}. 

Now for the second claim Lemma \ref{lem: d stab} gives that
 \begin{align*}
\abs{d_{G(\psi)} - d_{\wv}} \leq \norm{G(\psi) - \wv}_1 = \norm{\nabla{\Phi(\psi)}}_1 \leq \eta.
\end{align*}
In particular $d_{G(\psi)} \geq d_{\wv} - \eta$. The second result now follows from Theorem \ref{thm: spec bound}. 

\end{proof}

We briefly note that this proposition gives us uniqueness of the maximizer of $\Phi$ up to a multiple of $\onevect$. 

\begin{cor}\label{cor: unique max}
	
If $\psi_1, \psi_2$ are two maximizers of $\Phi$ then there is $r \in \R$ so that $\psi_1 - \psi_2 = r \onevect$. 
\end{cor}

\begin{proof}

Since $\Phi$ is concave and $\psi_1, \psi_2$ are two maximizers, $\Phi$ must be constant along the line segment joining $\psi_1$ and $\psi_2$. In particular if $t \in [0,1]$ then $\Phi(t \psi_1 + (1-t) \psi_2) = \Phi(\psi_1)$ and so $t \psi_1 + (1-t) \psi_2$ is a maximizer of $\Phi$. Hence $\nabla \Phi(t \psi_1 + (1-t) \psi_2) = 0$ for all $t \in [0,1]$. In particular we can conclude $[D^2\Phi(\psi_1)](\psi_2-\psi_1) = 0$. 

Now since $\psi_1$ is a maximizer we have $\nabla \Phi(\psi_1) = 0$ so the conditions of Proposition \ref{prop: well conditioned} are satisfied. Hence we have $D^2 \Phi(\psi_1) \leq -\frac{2^{3-1/q}}{C_\nabla M^{1/q}N^4\kappa}(d_{\wv} - \eta)^{1/q} P$ where $P$ is the orthogonal projection onto the hyperplane perpendicular to $\ov$. Hence we get $P(\psi_2-\psi_1) = 0$ which implies the claim. 

\end{proof}

\begin{rmk}\label{rmk: unique max constr}

Given any maximizer, $\psi$, of $\Phi$ we see that $\psi + r \onevect$ is also a maximizer. In particular given any maximizer, $\psi$, we can construct the maximizer $\psi_{max} := \psi - (\sum_{i=1}^N \psi^i) \onevect$ which is a maximizer of $\Phi$ that satisfies $\sum_{i=1}^N \psi_{max}^i = 0$. 

Conversely \ref{cor: unique max} tells us that there is a unique $\psi_{max}$ that maximizes $\Phi$ and satisfies $\sum_{i=1}^N \psi_{max}^i = 0$. For the remainder of the paper we shall refer to this unique maximizer as $\psi_{max}$. 

\end{rmk}

It is well-known that Proposition \ref{prop: well conditioned} gives local $1+\alpha$ convergence of the standard damped Newton Algorithm. However for convenience of the reader we will give a self contained proof. 

\begin{prop} \label{prop: Newt rate}
	
	If $\norm{\nabla{\Phi(\psi_0)}} < \frac{d_{\wv}}{2\sqrt{N}}$ then the iterates of Algorithm \ref{alg: damped newton} satisfy 
 	\begin{align*}
	\norm{\nabla \Phi(\psi_{k+1})} \leq (1 - \frac{\conj \tau_k}{2}) \norm{\nabla \Phi(\psi_{k})}
	\end{align*}
	where
 	\begin{align*}
	\conj \tau_k = \min(\frac{\gammaN^{1+\frac{1}{\alpha}}\delta}{L^{\frac{1}{\alpha}} \norm{{\nabla \Phi}(\psi)}2^{\frac{1}{\alpha}}}, 1)
	\end{align*}
	and
 	\begin{align*}
	\delta &:= \frac 12 \min \wv^i \\
	L &\leq C \delta^{-2} \\
	\gammaN &\geq \frac{2^{3-2/q}}{C_\nabla M^{1/q}N^4\kappa}d_{\wv}^{1/q}
	\end{align*}
	where $C$ is a constant depending only on $\norm{\rho}_\infty, \diam(X)$, and $c$ (the cost).

	Furthermore once we have $\conj \tau_k =1$ we get
 	\begin{align*}
	\norm{\nabla \Phi(\psi_{k+1})} \leq \frac{L\norm{{\nabla \Phi}(\psi)}^{1 + \alpha}}{\gammaN^{1+\alpha}}.
	\end{align*}
\end{prop}

\begin{proof}
	
	The proof is very similar to that of \cite[Proposition 6.1]{KitagawaMerigotThibert19}. However we will give all the details. 
	
	Note that $\norm{\nabla{\Phi(\psi_0)}} < \frac{d_{\wv}}{2\sqrt{N}}$ implies $\norm{\nabla{\Phi(\psi_0)}}_1 < \frac{d_{\wv}}{2}$. 
	
	For this proof define $\eps := \frac{d_{\wv}}{2}, \delta := \frac 12 \min \wv^i$. Note $\eps \leq \delta$ and so, by \eqref{eqn: J subset of K}, $\mathcal{J}^\eps \subset \mathcal{K}^\eps \subset \mathcal{K}^\delta$ (Note that it is NOT true in general that $\mathcal{J}^\eta \subset \mathcal{K}^\eta$; our choice of $\eps$ is the largest choice for which this is guaranteed to work). Furthermore $L$ is the $C^{1,\alpha}$ norm of $G$ on $\mathcal{K}^{\delta/2}$ and $-\gammaN$ is the bound of the second largest eigenvalue of $DG$ on $\mathcal{J}^{\eps}$. By Proposition \ref{prop: well conditioned} we have
 	 \begin{align*}
	 \gammaN 
	 \geq \frac{2^{3-1/q}}{C_\nabla M^{1/q}N^4\kappa}(\frac{d_{\wv}}{2})^{1/q} 
	 = \frac{2^{3-2/q}}{C_\nabla M^{1/q}N^4\kappa}d_{\wv}^{1/q}.
	 \end{align*}
	  Also by carefully tracing through the proofs in \cite{KitagawaMerigotThibert19} we have $L \leq C \delta^{-2}$. 
	
	We analyze a single iteration of the Algorithm. Define $\psi := \psi_k \in \mathcal{J}^\eps \subset \mathcal{K}^\eps \subset \mathcal{K}^\delta$.
	
	Let $v:= [D^2 \Phi(\psi)]^+ (\nabla \Phi(\psi))$. We see that
 	\begin{align*}
	\norm{v} \leq \frac{\norm{\nabla \Phi(\psi)}} {\gammaN},
	\end{align*}
	as $\Phi$ is $\gammaN$-concave in the direction orthogonal to $\onevect$.
	
	Also define $\psi_\tau = \psi - \tau v$. Let $\tau_1$ be the first exit time from $\mathcal{K}^{\delta/2}$. We have that
 	\begin{align*}
	\frac{\delta}{2} \leq \norm{G(\psi_{\tau_1}) - G(\psi)} \leq L \tau_1 \norm{v} \leq \frac{L}{\gammaN}\norm{\nabla \Phi(\psi)}\tau_1
	\end{align*}
	and so
 	\begin{align*}
	\tau_1 \geq \frac{\gammaN\delta}{2L\norm{{\nabla \Phi}(\psi)}}.
	\end{align*}
	Applying Taylor's formula to ${\nabla \Phi}$ we get
	
	\begin{equation}\label{eqn: ti Phi Taylor}
	{\nabla \Phi}(\psi_\tau) 
	= {\nabla \Phi}(\psi-\tau v) 
	= {\nabla \Phi}(\psi) - \tau (D{\nabla \Phi}(\psi))v + R(\tau)
	= {\nabla \Phi}(\psi) - \tau {\nabla \Phi}(\psi) + R(\tau)
	\end{equation}
	
	where
 	\begin{align*}
	\norm{R(\tau)}
	&= \norm{ \int_0^\tau (D{\nabla \Phi}(\psi_\sigma) - D{\nabla \Phi}(\psi))v  d\sigma } \\
	&= \norm{ \int_0^\tau DG(\psi_\sigma)v - DG(\psi)v d\sigma } \\
	&\leq \int_0^\tau L \norm{\psi_\sigma - \psi}^\alpha \norm{v} d\sigma\\
	&= \int_0^\tau L \norm{\sigma v}^\alpha \norm{v} d\sigma\\
	&= L \norm{v}^{\alpha+1} \int_0^\tau {\sigma}^\alpha d\sigma\\
	&= L \norm{v}^{\alpha+1} \frac{\tau^{\alpha + 1}}{\alpha+1}\\
	&\leq \frac{L\norm{{\nabla \Phi}(\psi)}^{1 + \alpha}}{\gammaN^{1+\alpha}}\tau^{1+\alpha} \numberthis \label{eqn: remainder estimate}
	\end{align*}
	for $\tau \leq 1$. 
	
	Finally we establish the error reduction estimates. \eqref{eqn: ti Phi Taylor} gives
 	\begin{align*}
	{\nabla \Phi}(\psi_\tau) 
	= (1-\tau){\nabla \Phi}(\psi) + R(\tau)
	\end{align*}
	so we have
 	\begin{align*}
	\norm{{\nabla \Phi}(\psi_\tau)} \leq (1-\frac{\tau}{2})\norm{{\nabla \Phi}(\psi)}
	\end{align*}
	provided that
 	\begin{align*}
	\norm{R(\tau)} \leq \frac{\tau}{2} \norm{{\nabla \Phi}(\psi)}.
	\end{align*}
	Again using \eqref{eqn: remainder estimate} this will be true provided
 	\begin{align*}
	\tau \leq \min(\tau_1, \frac{\gammaN^{1+\frac{1}{\alpha}}}{L^{\frac{1}{\alpha}} \norm{{\nabla \Phi}(\psi)}2^{\frac{1}{\alpha}}}) =: \tau_2.
	\end{align*}
	Hence we see that if we set $\conj \tau_k := \tau_2$, then the claim is true. Furthermore as the error goes to zero, eventually we must have $\tau_k = 1$. When this happens \eqref{eqn: ti Phi Taylor} gives
 	\begin{align*}
	{\nabla \Phi}(\psi_1) = R(1) 
	\end{align*}
	and so we get the super-linear convergence. 
	
\end{proof}

\section{Convergence of Gradient Descent Algorithm} \label{sec: convergence gradient}

In order to remedy the fact that the above algorithm only has local convergence we propose a regularized version in order to get within a close enough error.

We define the regularized Kantorovich's functional
 \begin{align*}
\ti \Phi(\psi) = \Phi(\psi) - \frac{\gamma}{2}\norm{\psi}^2 
\end{align*}
where $\gamma > 0$ is some parameter. 

We see that
 \begin{align*}
\nabla \ti \Phi(\psi) 
= \nabla \Phi(\psi) - \gamma \psi
= G(\psi) - \lambda - \gamma \psi
\end{align*}
and
 \begin{align*}
D^2 \ti \Phi(\psi) 
= D^2 \Phi(\psi) - \gamma I
= DG(\psi) - \gamma I. 
\end{align*}
In particular we note that $\ti \Phi(\psi)$ is strongly $\gamma$-concave. If we let $C_L$ be the Lipschitz constant of $G$ ($C_L = CN$ for some universal constant $C$) then $\ti \Phi$ has Lipschitz gradient with constant $C_L + \gamma$. Hence $\ti \Phi$ is well-conditioned. 

\begin{lem}
	
The projection of $\mathcal{K}^0$ in the direction $\onevect$ is bounded by
 \begin{align*}
M_0 := \sqrt{N} \( \max_{x\in X} (\max_{y \in Y} c(x,y) - \min_{y \in Y} c(x,y)) \).
\end{align*}
In other words if $\psi \in \mathcal{K}^0$ and $\sum_i \psi^i = 0$ then $\norm{\psi} \leq M_0$. 
	
\end{lem}

\begin{proof}

Since $\sum_i \psi^i = 0$ there are some indices $j_1, j_2 \in \{ 1, \dots, N \}$ so that $\psi^{j_1} \leq 0$ and $\psi^{j_2} \geq 0$. Suppose for sake of contradiction that $\norm{\psi} > M_0$. Then we must have $\norm{\psi}_\infty > \frac{M_0}{\sqrt{N}} = \max_{x\in X} (\max_{y \in Y} c(x,y) - \min_{y \in Y} c(x,y))$. In particular there is some index $j$ so that $\abs{\psi^j} > \max_{x\in X} (\max_{y \in Y} c(x,y) - \min_{y \in Y} c(x,y))$. 

We first look at the case where ${\psi^j} > \max_{x\in X} (\max_{y \in Y} c(x,y) - \min_{y \in Y} c(x,y))$. In this case
 \begin{align*}
\Lag_j(\psi) 
&= \{x\in X\mid c(x, y_j)+\psi^j= \min_{i} c(x,y_i) + \psi^i \} \\
&\subset \{x\in X\mid c(x, y_j)+\psi^j \leq c(x,y_{j_1}) + \psi^{j_1} \} \\
&\subset \{x\in X\mid \psi^j \leq c(x,y_{j_1}) -c(x, y_j) \}
= \emptyset
\end{align*}
and so we get $G(\psi)^j = 0$ which contradicts $\psi \in \mathcal{K}^0$. A similar argument handles the case where ${\psi^j} < - \max_{x\in X} (\max_{y \in Y} c(x,y) - \min_{y \in Y} c(x,y))$.

\end{proof}

We remark that since $X$ is compact and $c$ is continuous, $M_0 < +\infty$ and so the bound isn't vacuous. 

Since $\ti \Phi$ is strongly concave it has a unique maximizer. Although it may not be true that this maximizer is in $\mathcal{K}^0$ we show that it is still bounded. 

\begin{lem}\label{lem: unique max bound}
Let $\psi^*$ be the unique maximizer of $\ti \Phi$. Then $\norm{\psi^*} \leq M_0$. 
\end{lem}

\begin{proof}
	
Let $\psi_{max}$ be the maximizer of $\Phi$ constructed in Remark \ref{rmk: unique max constr} so that $\sum \psi_{max}^j = 0$. Then $G(\psi_{max}) = \wv$ and so $\psi_{max} \in \mathcal{K}^0$. In particular $\norm{\psi_{max}} \leq M_0$. But since $\psi^*$ is the maximizer of $\ti \Phi$ we have $\ti \Phi(\psi_{max}) \leq \ti \Phi(\psi^*)$. Hence:
 \begin{align*}
\Phi(\psi_{max}) - \gamma \norm{\psi_{max}}^2 &\leq \Phi(\psi^*) - \gamma \norm{\psi^*}^2 \\
- \gamma \norm{\psi_{max}}^2 &\leq - \gamma \norm{\psi^*}^2 \\
\norm{\psi_{max}}^2 &\geq \norm{\psi^*}^2
\end{align*}
and so the claim follows. 

\end{proof}

We now recall the standard gradient descent algorithm. 

\begin{algorithm}[H]
	
		\DontPrintSemicolon
		
		\SetKwInput{KwParameters}{Parameters}

		\KwParameters{A regularization constant $\gamma > 0$ and a step size $h > 0$.}
		
		\KwIn {A tolerance $\zeta > 0$ and an initial $\psi_0\in
		\R^N$ such that $\sum \psi_0^i = 0$.}
		
		\While{ $\norm{\gtp(\psi_k)}  \geq \zeta$}
		{
		\begin{description}
			\item[Step 1] Compute $\vec{d}_k = - \gtp(\psi_k)$
			\item [Step 2] Set $\psi_{k+1} = \psi_k + h\vec{d}_k$ and $k\gets k+1$.
		\end{description}
		}

	\label{alg: regularized gradient}
	\caption{Regularized Gradient Descent}
\end{algorithm}

It is well known that Algorithm \ref{alg: regularized gradient} converges with linear rate as $\ti \Phi$ is well conditioned. We recall the convergence rate.

\begin{prop}\label{prof: unique dual max}

Let $\psi^*$ be the unique maximizer of $\ti \Phi$. Suppose that $\gamma < \frac{\min \wv^i}{M_0}$. If we set $h = \frac{2}{C_L + 2\gamma}$ then the iterates of Algorithm \ref{alg: regularized gradient} satisfy 
 \begin{align*}
\norm{\psi_k - \psi^*} \leq \(\frac{Q_f - 1}{Q_f + 1}\)^k (M_0 + \norm{\psi_0})
\end{align*}
and
 \begin{align*}
\norm{\nabla \Phi(\psi_{k})} \leq  (C_L+\gamma) \(\frac{Q_f - 1}{Q_f + 1}\)^k (M_0 + \norm{\psi_0})
\end{align*}
where $Q_f = \frac{C_L + \gamma}{\gamma}$ is the condition number and $C_L$ is the Lipschitz constant of $G$. 

\end{prop} 

\begin{proof}

From the above lemma $\norm{\psi^*} \leq M_0$ hence $\norm{\psi_0 - \psi^*} \leq \norm{\psi_0} + \norm{\psi^*} \leq M_0 + \norm{\psi_0}$. The first claim now follows directly from \cite[Theorem 2.1.15]{Nesterov2018}. For the second claim since $\gtp$ is Lipschitz we have
 \begin{align*}
\norm{\gtp(\psi_{k}) - \gtp(\psi^*)} \leq (C_L+\gamma) \norm{\psi_k - \psi^*} \leq (C_L+\gamma) \(\frac{Q_f - 1}{Q_f + 1}\)^k (M_0 + \norm{\psi_0})
\end{align*}
and the result follows from recalling that $\gtp(\psi^*) = 0$ as $\psi^*$ is a maximizer. 

\end{proof}

Using the convergence rate we give a quick estimate of how many iterations it takes our algorithm to converge. 

\begin{cor}\label{cor: grad time}

With tolerance $\zeta$, parameter $\gamma$, starting input $\psi_0$, and step size $h = \frac{2}{C_L + 2\gamma}$, Algorithm \ref{alg: regularized gradient} terminates in at most $\log_{b} \frac{\zeta}{(M_0 + \norm{\psi_0}) (C_L+\gamma)}$ iterations where $b = \frac{C_L}{C_L + 2\gamma}$. 

\end{cor}

\begin{proof}

This follows from the second inequality in the above prop after noting that $b = \frac{Q_f - 1}{Q_f + 1}$. 

\end{proof}

We conclude this section with some quantitative stability bounds that estimate how much error our regularization introduced. 

\begin{prop}\label{prop: stab reg}
	
	Let $\ti \psi$ be the result of running Algorithm \ref{alg: regularized gradient} with tolerance $\zeta$, parameter $\gamma$, and starting input $\psi_0$. Then 
 	\begin{align*}
	\norm{G(\ti \psi) - \wv} < \zeta + \gamma (2M_0 + \norm{\psi_0})
	\end{align*} 
	where $M_0$ is as defined above.
	
\end{prop}

\begin{proof}
	
We have
 \begin{align*}
G(\ti \psi) - \wv 
= \gtp(\ti \psi) + \gamma \ti \psi
= \gtp(\ti \psi) + \gamma (\ti \psi - \psi^* + \psi^*),
\end{align*}
where $\psi^*$ is the unique maximizer of $\ti \Phi$. By Proposition \ref{prof: unique dual max} we have $\norm{\ti \psi - \psi^*} \leq M_0 + \norm{\psi_0}$ and by Lemma \ref{lem: unique max bound} we have $\norm{\psi^*} \leq M_0$. Hence we get
 \begin{align*}
\norm{G(\ti \psi) - \wv} 
\leq \norm{\gtp(\ti \psi)} + \gamma (\norm{\ti \psi - \psi^*} + \norm{\psi^*})
< \zeta + \gamma (2M_0 + \norm{\psi_0})
\end{align*}
as desired.

\end{proof}

\section{Two-Stage Algorithm} \label{sec: two stage}

We now propose a two-stage algorithm.

\begin{algorithm}[H]
	
		\DontPrintSemicolon

		\KwIn{ A tolerance $\ti \zeta > 0$.}
		
		\begin{description}
			\item[Stage 1] Run Algorithm \ref{alg: regularized gradient} with tolerance $\zeta = \frac{d_{\wv}}{4\sqrt{N}}$, parameter $\gamma = \frac{d_{\wv}}{8M_0\sqrt{N}}$, starting input $\psi_ 0 = \bm{0}$, and step size $h = \frac{2}{C_L + 2\gamma}$. Let $\ti \psi$ be the output. 
			\item [Stage 2] Run Algorithm \ref{alg: damped newton} with parameters $\zeta = \ti \zeta$ and $\psi_ 0 = \ti \psi$. 
		\end{description}

	\label{alg: two stage}
	\caption{Two Stage Algorithm}
\end{algorithm}

\begin{prop}

Stage 1 of Algorithm \ref{alg: two stage} terminates in at most $\frac{32 N M_0^2 C_L^2}{d_{\wv}^2}$ steps (note that this is a constant independent of $\ti \zeta$).

Let $E = \frac{\gammaN^{1+\frac{1}{\alpha}}\delta}{L^{\frac{1}{\alpha}} 2^{\frac{1}{\alpha}}}$. If $\ti \zeta \geq E$ then Stage 2 terminates in at most
 \begin{align*}
\max\( \frac{2(\frac{d_{\wv}}{2\sqrt{N}} - \ti \zeta)}{E}, 0 \)
\end{align*}
steps. 

Otherwise Stage 2 terminates in at most
 \begin{align*}
\max\( \frac{2(\frac{d_{\wv}}{4\sqrt{N}} - E)}{E}, 0 \) + \log_{1+\alpha} \frac{\log \frac{L^{\frac{1}{\alpha}}}{\gammaN^{1+\frac{1}{\alpha}}}\ti \zeta}{\log \frac{L^{\frac{1}{\alpha}}}{\gammaN^{1+\frac{1}{\alpha}}}E}
\end{align*} 
 steps.

\end{prop}	

\begin{proof}

For the first claim we have using Corollary \ref{cor: grad time} that the maximum number of iterations is
 \begin{align*}
\frac{\log \frac{\zeta}{M_0(C_L + \gamma)}}{\log \frac{C_L}{C_L + 2\gamma}}
= \frac{\log \frac{\frac{d_{\wv}}{4\sqrt{N}}}{M_0(C_L + \frac{d_{\wv}}{8M_0\sqrt{N}})}}{\log \frac{C_L}{C_L + 2\frac{d_{\wv}}{8M_0\sqrt{N}}}}
= \frac{\log \frac{M_0(C_L + \frac{d_{\wv}}{8M_0\sqrt{N}})}{\frac{d_{\wv}}{4\sqrt{N}}}}{\log \frac{C_L + 2\frac{d_{\wv}}{8M_0\sqrt{N}}}{C_L}}
= \frac{\log (\frac 12 + \frac{4\sqrt{N}M_0C_L}{d_{\wv}})}{\log (1 + \frac{d_{\wv}}{4\sqrt{N}M_0C_L})}.
\end{align*}
Now using the classical inequalities, $\log(1+x) \geq \frac{x}{1+x}$ and $\log(\frac 12 +x)\leq \log(1+x) \leq x$ we get
\begin{align*}
\frac{\log (\frac 12 + \frac{4\sqrt{N}M_0C_L}{d_{\wv}})}{\log (1 + \frac{d_{\wv}}{4\sqrt{N}M_0C_L})}
\leq \frac{\frac{4\sqrt{N}M_0C_L}{d_{\wv}}}{\(\frac{\frac{d_{\wv}}{4\sqrt{N}M_0C_L}}{1 + \frac{d_{\wv}}{4\sqrt{N}M_0C_L}}\)}
\leq \frac{\frac{4\sqrt{N}M_0C_L}{d_{\wv}}}{\frac{\frac{d_{\wv}}{4\sqrt{N}M_0C_L}}{2}}
= \frac{32 N M_0^2 C_L^2}{d_{\wv}^2},
\end{align*}
where we have used $M_0, C_L \geq 1$ and $d_{\wv} \leq 1$ for the last inequality.

Next by Proposition \ref{prop: stab reg} we see that
 \begin{align*}
\norm{\nabla \Phi(\ti \psi)} = \norm{G(\ti \psi) - \wv} < \frac{d_{\wv}}{4\sqrt{N}} + \frac{d_{\wv}}{8M_0\sqrt{N}} (2M_0 + \norm{\psi_0}) = \frac{d_{\wv}}{2\sqrt{N}} 
\end{align*} 
where $\ti \psi$ is the initial input for Stage 2 in Algorithm \ref{alg: two stage}. Hence the result of Proposition \ref{prop: Newt rate} applies. In particular as long as $\frac{E}{\norm{\nabla \Phi(\psi_k)}} \leq 1$ we have
 \begin{align*}
\norm{\nabla \Phi(\psi_{k+1})} 
\leq (1 - \frac{\conj \tau_k}{2}) \norm{\nabla \Phi(\psi_{k})} 
= (1 - \frac{E}{2\norm{\nabla \Phi(\psi_k)}}) \norm{\nabla \Phi(\psi_{k})}
= \norm{\nabla \Phi(\psi_k)} - \frac{E}{2}
\end{align*}
which shows the second claim. 

Finally once we have $\norm{\nabla \Phi(\psi_k)} < E$, we see that $\conj \tau_k = 1$ and so by Proposition \ref{prop: Newt rate} we get
 \begin{align*}
\norm{\nabla \Phi(\psi_{k+1})} \leq \frac{L\norm{{\nabla \Phi}(\psi_k)}^{1 + \alpha}}{\gammaN^{1+\alpha}}
\end{align*}
and so 
 \begin{align*}
\norm{\nabla \Phi(\psi_{k+j})} 
\leq \(\frac{L}{\gammaN^{1+\alpha}}\)^{\frac{(1+\alpha)^j - 1}{\alpha}} \norm{\nabla \Phi(\psi_{k})}^{(1+\alpha)^j}
< \(\frac{L}{\gammaN^{1+\alpha}}\)^{\frac{(1+\alpha)^j - 1}{\alpha}} E^{(1+\alpha)^j}.
\end{align*}
Hence if
 \begin{align*}
j 
\geq \log_{1+\alpha} \frac{\log \frac{L^{\frac{1}{\alpha}}}{\gammaN^{1+\frac{1}{\alpha}}}\ti \zeta}{\log \frac{L^{\frac{1}{\alpha}}}{\gammaN^{1+\frac{1}{\alpha}}}E}
\end{align*}
then we will have $\norm{\nabla \Phi(\psi_{k+j})} < \ti \zeta$ as desired. This completes the proof of the third and final claim. 

\end{proof}

\section{Convergence of Cells} \label{sec: convergence cells}

In this section we obtain quantitative convergence of the Laguerre cells associated with our problem. We recall $\psi_{max}$, which is the maximizer of $\Phi$ constructed in Remark \ref{rmk: unique max constr} so that $\sum \psi_{max}^j = 0$.

The result on the $\mu$-symmetric convergence of the cells follows immediately from \cite{BansilKitagawa19b}.

\begin{prop}
 \begin{align*}
\sum_{i=1}^N \Delta_\mu({\Lag_{i}(\psi)}, {\Lag_{i}(\psi_{max})}) \leq 4 N \norm{\nabla \Phi(\psi)}_1
\end{align*}
\end{prop}

\begin{proof}

This is an immediate consequence of \cite[Corollary 5.7]{BansilKitagawa19b} after applying Remark 5.8 there. 

\end{proof}

The Hausdorff convergence result in \cite{BansilKitagawa19b} strongly depends on the global PW inequality. Hence we have to do more work in our case to recover the result. 

\begin{thm}
	
Let $\psi \in \R^N$ be so that $\inner{\psi}{\onevect} = 0$. 
Suppose that 
\begin{align*}
\norm{\nabla \Phi(\psi)} < \min (\frac{2^{3-2/q}d_{\wv}^{1/q} \delta}{\norm{\rho}_\infty C_\Delta C_\nabla M^{1/q}N^5\kappa}, \frac{d_{\wv}}{2\sqrt{N}})
\end{align*}
 where $\delta = \frac 12 \min \wv^i$. Then
 \begin{align*}
\norm{\psi_{max} - \psi} \leq \frac{C_\nabla M^{1/q}N^4\kappa}{2^{3-2/q}d_{\wv}^{1/q}} \norm{\nabla \Phi(\psi)}
\end{align*}
and
 \begin{align*}
d_\H({\Lag_{i}(\psi_{max})}, {\Lag_{i}(\psi)} )^n 
\leq  \frac{C_1 C_\nabla M^{1/q}N^5\kappa}{2^{3-2/q}d_{\wv}^{1/q}\left(\arccos(1-C_2\L({\Lag_{i}(\psi_{max})})^2)\right)^{n-1}} \norm {\nabla \Phi(\psi)} 
\end{align*}
where $C_1, C_2$ are the constants described in \cite[Theorem 6.6]{BansilKitagawa19b}. 
\end{thm}

\begin{proof}
	
Setting $\eps := \frac{d_{\wv}}{2}$ we see that $\norm{\nabla \Phi(\psi)} < \frac{d_{\wv}}{2\sqrt{N}}$ implies that $\psi \in \mathcal{J}^\eps$, where $\mathcal{J}^\eps$ is as in Proposition \ref{prop: well conditioned}. We recall by Proposition \ref{prop: well conditioned} that $\Phi$ is strongly concave on $\mathcal{J}^\eps$ except in the direction $\onevect$ with parameter $\frac{2^{3-2/q}}{C_\nabla M^{1/q}N^4\kappa}d_{\wv}^{1/q}$. Since we have that $\psi$ and $\psi_{max}$ are normal to $\onevect$ we get
\begin{align*}
\inner{\nabla \Phi(\psi)}{\psi - \psi_{max}} 
&= \int_0^1 \inner{D^2\Phi[(1-t)\psi_{max} + t\psi](\psi-\psi_{max})}{\psi-\psi_{max}}dt \\
&\leq \frac{-2^{3-2/q}d_{\wv}^{1/q}}{C_\nabla M^{1/q}N^4\kappa} \norm{\psi - \psi_{max}}^2.
\end{align*}
From the Cauchy-Schwarz inequality we obtain $\norm {\nabla \Phi(\psi)} \geq \frac{2^{3-2/q}d_{\wv}^{1/q}}{C_\nabla M^{1/q}N^4\kappa} \norm{\psi - \psi_{max}}$. Hence $\norm{\psi - \psi_{max}} \leq \frac{C_\nabla M^{1/q}N^4\kappa}{2^{3-2/q}d_{\wv}^{1/q}}\norm {\nabla \Phi(\psi)}$. Now using the assumed bound on $\norm {\nabla \Phi(\psi)}$ we get
 \begin{align*}
\norm{\psi - \psi_{max}}_\infty 
\leq \norm{\psi - \psi_{max}} 
< \frac{C_\nabla M^{1/q}N^4\kappa}{2^{3-2/q}d_{\wv}^{1/q}}(\frac{2^{3-2/q}d_{\wv}^{1/q} \delta}{\norm{\rho}_\infty C_\Delta C_\nabla M^{1/q}N^5\kappa})
= \frac{\delta}{\norm{\rho}_\infty C_\Delta N}. 
\end{align*}
Since $\wv^i  \leq \norm{\rho}_\infty\L({\Lag_{i}(\psi_{max})})$ we get $\norm{\psi - \psi_{max}}_\infty < \frac{\min_i \L({\Lag_{i}(\psi_{max})})}{2C_\Delta N}$. Hence the conditions of \cite[Theorem 6.6]{BansilKitagawa19b} are satisfied and we get
 \begin{align*}
d_\H({\Lag_{i}(\psi_{max})}, {\Lag_{i}(\psi)} )^n 
&\leq \frac{C_1N\norm{\psi_{max} - \psi}_\infty}{\left(\arccos(1-C_2\L({\Lag_{i}(\psi_{max})})^2)\right)^{n-1}} \\
&\leq \frac{C_1 C_\nabla M^{1/q}N^5\kappa}{2^{3-2/q}d_{\wv}^{1/q}\left(\arccos(1-C_2\L({\Lag_{i}(\psi_{max})})^2)\right)^{n-1}} \norm {\nabla \Phi(\psi)}
\end{align*}
as desired. 

\end{proof}

\begin{rmk}
This shows that Algorithm \ref{alg: two stage} has local superlinear convergence even when the error is measured by the Hausdorff distances between the cells. Furthermore, the size of the zone of local convergence is explicitly determined from the initial data. 
\end{rmk}

\section{Estimation of $d_{\wv}$} \label{sec: estimation of d}

In this section we give some methods to estimate or compute $d_{\wv}$. Recall our notation, $\dv^j = \mu(X_j)$. 

First we prove the claim made in Remark \ref{rmk: small set}. 

\begin{prop}\label{prop: small set}
	
Let $X_1, \dots, X_M$ be any fixed decomposition of $X$. Let $B := \{ \wv \in \wvs : d_\wv = 0 \} = \{ \wv \in \wvs: \subprop{\wv} \cap S_\dv \neq \emptyset \}$. Then $\dim B = N-2$ where dimension is understood in the sense of Hausdorff dimension. 
	
\end{prop}

\begin{proof}
	
Note that since $\dv$ is a vector with a finite number of entries, $S_\dv$ is a finite set. Say $S_\dv = \{s_1, \dots, s_l\}$. We denote the set of all length $N$ vectors in $\R^N$ whose entries are all $0$ or $1$ by $\F_2^{N}$. We let $\F_2^{N, prop} = \F_2^N \setminus \{(0, \dots, 0), (1, \dots, 1)\}$. 

Now for any vector $v \in \F_2^{N}$ and $i \in \{1, \dots, l\}$ define $A_{v,i} = \{\wv \in \R^N: \inner{v}{\wv} = s_i \}$. We see that
 \begin{align*}
B = \bigcup_{(v,i) \in \F_2^{N, prop} \times \{1, \dots, l \}} A_{v,i} \cap \wvs
\end{align*}
as for any $v \in \F_2^N$, 
 \begin{align*}
A_{v,i} \cap \wvs = \{\wv \in \wvs: \sum_{\{j: v^j = 1\}} \wv^j = s_i\}.
\end{align*}
Note that as long as $v \neq (0, \dots, 0)$, $A_{v,i}$ is a hyperplane normal to $v$. Since $\wvs$ is contained in a hyperplane normal to $(1, \dots, 1)$ we see that for every $v \in \F_2^{N, prop}$, $A_{v,i}$ is a hyperplane transversal to $\wvs$ (as $v \neq (1, \dots, 1)$) and hence $A_{v,i} \cap \wvs$ has dimension at most $N-2$. Since $B$ is a finite union of sets of this form we see that $B$ has dimension at most $N-2$. 

Note that since $B$ always contains $\{\wv \in \wvs: \wv^1 = 0\}$, $\dim B \geq N-2$. 
	
\end{proof}

Next we examine the case where $\wv^i$ and $\dv^j$ are all rational numbers and the denominators of $\wv^i$ are relatively prime to those of $\dv^j$. Note that the probability of two large integers being relatively prime is $\frac{\pi^2}{6} \approx 60\%$.

\begin{prop}

Suppose that there are relatively prime positive integers $\ti N, \ti M$ so that $\ti N \wv^i, \ti M \dv^j \in \N$. Then $d_{\wv} \geq \frac{1}{\ti M \ti N}$.

\end{prop}

\begin{proof}

Let $I, J$ be the minimizing index sets so that
 \begin{align*}
d_{\wv} = \abs{\sum_{i \in I} \wv^i - \sum_{j \in J} \dv^j}
\end{align*}
where $I \neq \emptyset, \{1, \dots, N \}$. 

Then
 \begin{align*}
\ti M \ti N d_{\wv} = \abs{\ti M \sum_{i \in I} \ti N \wv^i - \ti N \sum_{j \in J} \ti M \dv^j}.
\end{align*}
We claim that $\ti M \sum_{i \in I} \ti N \wv^i \neq \ti N \sum_{j \in J} \ti M \dv^j$. If it did then we would have that $\ti N$ divides $\ti M \sum_{i \in I} \ti N \wv^i$. Since $\ti M, \ti N$ are relatively prime this gives us that $\ti N$ divides $\sum_{i \in I} \ti N \wv^i$. In particular either $\sum_{i \in I} \ti N \wv^i = 0$ or $\sum_{i \in I} \ti N \wv^i \geq \ti N$ which would imply either $\sum_{i \in I} \wv^i = 0$ or $\sum_{i \in I} \wv^i \geq 1$ respectively. The first case gives $I = \emptyset$ and the second gives $I = \{1, \dots, N \}$ both of which are contradictions. 

Now since we have $\ti M \sum_{i \in I} \ti N \wv^i \neq \ti N \sum_{j \in J} \ti M \dv^j$ we get $\ti M \sum_{i \in I} \ti N \wv^i - \ti N \sum_{j \in J} \ti M \dv^j \neq 0$. Since this is an integer we must then have $\abs{\ti M \sum_{i \in I} \ti N \wv^i - \ti N \sum_{j \in J} \ti M \dv^j} \geq 1$. Hence
 \begin{align*}
\ti M \ti N d_{\wv} = \abs{\ti M \sum_{i \in I} \ti N \wv^i - \ti N \sum_{j \in J} \ti M \dv^j} \geq 1
\end{align*}
and so $d_{\wv} \geq \frac{1}{\ti M \ti N}$ as desired.

\end{proof}

One particularly useful case of this is when we have equal weights.

\begin{cor}
	
If we have $\gcd(M,N) = 1$, $\lambda^i = \frac{1}{N}$, and $\dv^j = \frac{1}{M}$ then $d_{\wv} \geq \frac{1}{MN}$. 
\end{cor}

In general we can see that computation of $d_{\wv}$ is equivalent to a variant of the subset sum problem. 

\begin{prop}\label{prop: subset sum equiv}

Let $S = \{ \wv^2, \dots, \wv^N, -\dv^1, \dots, -\dv^M \}$ (note that $S$ doesn't include $\wv^1$). Then
\begin{align*}
d_{\wv} = \min_{A\subset S} \abs{ \sum_{a\in A} a }
\end{align*}
where the minimum is taken over subsets $A$ of $S$ that contain some $\wv^i$. 

\end{prop}

\begin{proof}

Let $\mbf{\Lambda} = \{\wv^1, \dots, \wv^N \}$ and $\mbf{X} = \{-\dv^1, \dots, -\dv^M \}$.

Note that if $A \subset S$ and $A \cap \mbf{\Lambda} \neq \emptyset$ then
 \begin{align*}
\abs{\sum_{a\in A} a} = \abs{\sum_{p \in A \cap \mbf{\Lambda}} p + \sum_{q \in A \cap \mbf{X}} q} \geq d_{\wv}
\end{align*}
as $\sum_{q \in A \cap \mbf{\Lambda}} q + \sum_{p \in A \cap \mbf{X}} p$ is the difference between a subset sum of $\wv$ and $\dv$. 

Conversely suppose that we are given $P \subset \mbf{\Lambda}$ and $Q \subset \{\dv^1, \dots, \dv^M \}$ where $P \neq \emptyset, \mbf{\Lambda}$ such that $d_{\wv} = \abs{\sum_{p \in P} p - \sum_{q \in Q} q}$. If $\wv^1 \not\in P$ set $A = P \cup (-Q)$, where $-Q = \{-q: q \in Q \}$. We see that
 \begin{align*}
\abs{\sum_{a\in A} a} = \abs{\sum_{p \in P} p - \sum_{q \in Q} q} = d_{\wv}.
\end{align*}
If $\wv^1 \in P$ then set $A = (\mbf{\Lambda} \setminus P) \cup  (\mbf{X} \setminus -Q)$.

We see that
 \begin{align*}
\abs{\sum_{a\in A} a} 
= \abs{\sum_{p \not\in P} p - \sum_{q \not\in Q} q} 
= \abs{(1 - \sum_{p \in P} p) - (1 - \sum_{q \in Q} q)}
=\abs{\sum_{p \in P} p - \sum_{q \in Q} q}
= d_{\wv}.
\end{align*}
In either case we will have 
$
\min_{A\subset S} \abs{ \sum_{a\in A} a } \leq d_{\wv}
$
where again the minimum is taken over subsets $A$ of $S$ that contain some $\wv^i$.

\end{proof}

\begin{rmk}

Unfortunately the subset sum problem is NP complete and so in general it may not be feasible to compute the exact value of $d_{\wv}$. However in order to use the results of this paper all one needs is an approximation to $d_{\wv}$. Although there are fast approximate algorithms for the subset sum problem and an extensive literature, it seems upon inspection that most papers require the elements of the set to be positive integers or put other assumptions that are unreasonable for our situation. However, a slight variation to the classical approximate subset sum algorithm works for our case. We present it for the convenience of the reader.

\end{rmk}

First we describe a ``Trim'' function. This function takes an ordered list of real numbers and throws out almost duplicates. Our Trim function differs from the classical Trim function only in that it measures errors in absolute scale rather than log scale. 

\begin{algorithm}[H]	
		\SetKwInput{KwParameters}{Parameters}
		
		\DontPrintSemicolon	
	
		\KwParameters{ A tolerance $\eps > 0$.}
		
		\KwIn{A sorted list of real numbers, $\ti A = ( r_1, \dots, r_m )$, where $r_1 \geq r_2 \geq \dots \geq r_m$.}
		
		\Begin
		{
		$ A = (r_1) $ \;
		
		$last = r_1$ \;
		
		\For {$i \in \{2, \dots, m\}$} 
		{		
			\If {$r_i \geq last + {\eps}$} 
{			
			$last = r_i$ \;
			
			Append $r_i$ to $ A$ \;}
		}
		
		\Return {$A$}
}
	\label{alg: trim}
	\caption{Trim Routine}
\end{algorithm}

It is clear that Algorithm $\ref{alg: trim}$ runs in $O(m)$ time where $m = \abs{A}$, the size of the input list. With the modified Trim function our approximate subset sum algorithm is virtually identical to the classical one. We also recall that given two sorted lists $L_1, L_2$, we can compute the sorted list given by their union in $O(\abs{L_1} + \abs{L_2})$ time (see, for example, \cite[Section 2.3.1]{Cormen09}). Furthermore for a list $L$ and $r \in \R$ we use the notation $L+r$ to denote the list obtained from adding $r$ to each element of $L$. 

\begin{algorithm}[H]
	
	\SetKwInput{KwParameters}{Parameters}
	
	\DontPrintSemicolon	
	
		\KwParameters {A tolerance $\eps > 0$.}
		
		\KwIn {An ordered list $S = (\wv^2, \dots, \wv^N, -\dv^1, \dots, -\dv^M)$, with $\wv^i, \dv^j > 0$. }
		
		\Begin
	{	
		$L_1 = ( 0 )$ \;
		
		\For {$i \in \{2, \dots, N\}$}
{		
			$\ti L_i = L_{i-1} \cup (L_{i - 1} + \wv^i)$ \;
			 
			$L_i$ =  Trim($\ti L_i$, $\frac{\eps}{N+M}$) \;	
	
}
		
		$L_N = L_N \setminus \{0 \}$. \;
	
		\For {$i \in \{N+1, \dots, N + M\}$}
{
			$\ti L_i = L_{i-1} \cup (L_{i - 1} - \dv^{i-N})$ \;
			
			$L_i$ = Trim($\ti L_i$, $\frac{\eps}{N+M}$)	\;

}
	\Return $\min_{l \in L_{N+M}} \abs{l}$
}
	\label{alg: subset sum}
	\caption{Approximate Subset Sum}
\end{algorithm}

An analysis similar to that of the classical approximate subset sum algorithm (see for example, \cite[Theorem 35.8]{Cormen09}) shows that Algorithm \ref{alg: subset sum} runs in polynomial time and indeed approximates $d_{\wv}$. For the convince of the reader we sketch the proof.

\begin{prop}

Algorithm \ref{alg: subset sum} terminates in at most $O(\frac{(N+M)^2}{\eps})$ steps. Furthermore if $x$ is the value returned by the algorithm then $x - \eps \leq d_{\wv} \leq x$.

\end{prop}

\begin{proof}
	
Note that if $y \in \ti L_i$ for any $i$ then $y \in [-1,1]$. Hence we see by construction of our Trim Routine (Algorithm \ref{alg: trim}), that $\abs{L_i} \leq \frac{2(N+M)}{\eps}$, where $\abs{L_i}$ denotes the number of elements in $L_i$. This proves that Algorithm \ref{alg: subset sum} terminates in at most $O(\frac{(N+M)^2}{\eps})$ steps.

Next define $\ha L_i$ to be the collection of all subset sums of the first $i-1$ elements of $S$ when $i \leq N-1$ and define $\ha L_i$ to be the collection of all subset sums of the first $i-1$ elements of $S$ that contain some $\wv^j$ when $i > N-1$. It is clear that $L_i \subset \ha L_i$. In particular $L_{N+M} \subset \ha L_{N+M}$. Since by Proposition \ref{prop: subset sum equiv}, $d_\wv = \min_{y \in \ha L_{N+M}} \abs{y}$, this gives us that $d_\wv \leq x$. 

Next note that by construction of our Trim Routine (Algorithm \ref{alg: trim}), for every $y \in \ti L_i$ there is a $z \in L_i$ so that $\abs{y-z} \leq \frac{\eps}{M+N}$, i.e. $d_\H(\ti L_i, L_i) \leq \frac{\eps}{M+N}$, where we use $d_\H$ for Hausdorff distance. Since $L_i$ arises from $i$ applications of the Trim routine we see that $d_\H(\ha L_i, L_i) \leq i\frac{\eps}{M+N}$. In particular $d_\H(\ha L_{M+N}, L_{M+N}) \leq ({M+N})\frac{\eps}{M+N} = \eps$ and so by Proposition \ref{prop: subset sum equiv}, we get $d_\wv \geq x - \eps$.
	
\end{proof}

\section{Acknowledgments}

I would like to thank Prof. Kitagawa for many helpful comments and suggestions on a previous version of this manuscript. 

\bibliographystyle{alpha}
\bibliography{snowshovelingalg}

\end{document}